\documentclass[preprint, 12pt]{article}

\pdfoutput=1

\usepackage{amsmath,amsthm,amsfonts,amssymb,amscd}
\usepackage{graphicx,bm,color}
\usepackage{algorithm}
\usepackage[algo2e]{algorithm2e}
\usepackage{algpseudocode}
\usepackage{stmaryrd}
\usepackage{listings}
\usepackage{lineno}
\usepackage{mathptmx}
\usepackage{cleveref}
\usepackage{tikz}
\usepackage[skins,theorems]{tcolorbox}
\usepackage{authblk}
\usepackage{tkz-graph}
\usepackage[latin1]{inputenc}
\usetikzlibrary{shapes,arrows}
\usepackage{smartdiagram}
\usetikzlibrary{arrows,calc, decorations.markings, intersections}
\usepackage{caption}
\usepackage{cancel}
\usepackage{mathtools}
\usepackage{url}
 
\definecolor{gray}{gray}{0.6}

\theoremstyle{plain}
\newtheorem{thm}{Theorem} 

\theoremstyle{remark}
\newtheorem{remark}{Remark}
\newtheorem*{remark*}{Remark}

\newcommand{\vertiii}[1]{{\left\vert\kern-0.25ex\left\vert\kern-0.25ex\left\vert #1 
		\right\vert\kern-0.25ex\right\vert\kern-0.25ex\right\vert}}

\newcommand\norm[1]{\left\lVert#1\right\rVert}

\begin{document}

\title{A \textit{priori} error analysis for transient problems using Enhanced Velocity approach in the discrete-time setting.}


\author[1,2]{Yerlan Amanbek}
\author[1]{Mary F. Wheeler}
\affil[1]{Center for Subsurface Modeling, Institute for Computational Engineering and Sciences, University of Texas at Austin}
\affil[2]{School of Science and Technology, Nazarbayev University}
\affil[ ]{\textit{yerlan.amanbek@nu.edu.kz, mfw@ices.utexas.edu}}


\date{\today}
\maketitle

\begin{abstract}
Time discretization along with space discretization is important in the numerical simulation of subsurface flow applications for long run.
In this paper, we derive theoretical convergence error estimates in discrete-time setting for transient problems with the Dirichlet boundary condition.
Enhanced Velocity Mixed FEM as domain decomposition method is used in the space discretization  and the backward Euler method and the Crank-Nicolson method are considered in the discrete-time setting. Enhanced Velocity scheme was used in the adaptive mesh refinement dealing with heterogeneous porous media \cite{amanbek2017adaptive,singh2017adaptive} for single phase flow and transport and demonstrated as mass conservative and efficient method. 
Numerical tests validating the backward Euler theory are presented.
This error estimates are useful in the determining of time step size and the space discretization size.
\end{abstract}




\newcommand{\bs}[1]{\boldsymbol{#1}}

\section{Introduction}
The most subsurface flow equations are dynamic and time-dependent problems. In decision making process, numerical simulation of flow plays vital role in many engineering applications such as oil and gas production evaluation, CO$_2$ sequestration and contaminate transport problems. It is natural to deal with non-matching multiblock grids in the reservoir simulation since subsurface parameters such as permeability or porosity can vary over subdomains substantially. The accuracy of simulation can depend on discretization method of space and time variables. For space discretization, we are concerned with a well-established domain decomposition method, i.e.  Enhanced Velocity Mixed Finite Element Method (EVMFEM), which provides similar accuracy as the Multiscale Mortar Mixed FEM \cite{wheeler2002enhanced}. EVMFEM is a mass conservative and an efficient domain decomposition method. By using this method, several applications such as single, two-phase flow, bio-remediation simulation and others were considered in \cite{wheeler2002enhanced, thomas2011enhanced}. Recently, an  adaptive mesh refinement strategy, which is based on Enhanced Velocity scheme, has been proposed in the numerical simulations of flow and transport through heterogeneous   porous media \cite{amanbek2017adaptive, singh2017adaptive,  ganis2018adaptive,amanbek2018new}. Such a novel approach demonstrates the efficiency and accuracy of simulation in the heterogeneous porous media by allowing to capture an important features of flow and transport problems. 

A little attention has been given to discrete time setting analysis. Theoretical convergence analysis of EVMFEM has been shown in \cite{thomas2011enhanced} for slightly compressible single phase flow for general continuous in time approximations. However, we could not find the numerical error tests that compare with analytical solution. A few authors have begun to implement time domain decomposition method to be flexible on selection of time step size \cite{singh2018space, singh2018domain, Amanbek2018}. The key idea is to extend the EVMFEM in space  to time discretization by constructing a monolithic system without subdomain iteration.
 
In this paper, we are concerned with the solution of time dependent problem that is discretized by the backward Euler method or the Crank-Nicolson method combined with EVMFEM, which uses the lowest order Raviart-Thomas spaces on non-matching multiple subdomains. In particularly, we focus on deriving a \textit{priori} error estimates for the transient subsurface problems in discrete-time setting. This gives asymptotive behavior of the numerical error for a given mesh size, time step size and others. This analsys allows us to conclude about the convergence and the guarantee of stability of the numerical method. The reader is referred to \cite{thomee1984galerkin,wheeler1973priori} for more different time and space discretization in the transient problems.  We also provide numerical tests of the error estimate.

This paper is organized as follows. In the next section, we describe the slightly compressible flow model formulation as well as give a strong and weak formulation using Enhanced Velocity space. Section 3 is devoted to the error analysis with preliminary projections, definitions and discrete formulation. This analysis carried out for backward Euler scheme and Crank-Nicolson schemes simultaneously in time discretization settings. Numerical results are presented in Section 4; Section 5 concludes the paper.
 
\section{Model Formulation} \label{sec:model}
We describe in this section the slightly compressible flow formulate with initial and boundary conditions. Next, the transient problem is presented in the strong and weak formulations.
 
\subsection{Slightly Compressible Flow Formulation}
Our focus is a single phase and slightly compressible fluid in heterogeneous porous media. The classical mass conservation equation is defined by 
	\begin{equation} \label{eq:2_1}
	\frac{\partial }{\partial t}\left(\phi \rho \right)+ \nabla \cdot (\rho \mathbf{u}) =q \qquad \qquad \text{in} \quad \Omega \times J
	\end{equation}
where  $\Omega \in \mathbb{R}^d( d=1,2$ or $3$), $J=(0,T]$, $d$ is the number of spatial dimensions, $q$ is the source/sink term, $\phi$ is the porosity, $\rho$ is the phase density, and  $\textbf{u}$ is the phase velocity. We remark that a Peacemann correction is used for modeling source/sink terms \cite{peaceman1983interpretation}. 

In slightly compressible fluid, the phase density is given by 
$ \rho=\rho_{ref} e^{C_f(p-p_{ref})} \label{eq:2_2}$, where where, $C_f$ is the fluid compressibility, and $\rho_{ref}$ is the reference density at reference pressure $p_{ref}$. Using Taylor series expansion we obtain $\rho \approx \rho_{ref}(1+C_f(p-p_{ref}))$. 
Then it follows that 
\begin{equation}
	\phi \frac{\partial \rho }{\partial t}=\phi\frac{\partial \rho }{\partial p} \frac{\partial p }{\partial t}=\phi C_1 \frac{\partial p }{\partial t}
\end{equation} for invariant-in-time $\phi$ and for $C_1 = C_f \rho_{ref}$. 
The phase velocity $\textbf{u}$ is defined by Darcy's law as,
\begin{equation} 
\textbf{u} = - \frac{\mathbf{K}}{\mu}\left(\nabla p - \rho \textbf{g} \right),
\label{eq:darcy}
\end{equation}
where, $\mu$ is the viscosity, $\mathbf{K}$ is the permeability (absolute permeability) tensor, $\rho$ is the density of the fluid and $\textbf{g}$ is the gravity vector. Although more general global boundary conditions can also be treated, we restrict ourselves to the following,
\begin{equation}
\begin{aligned}
p = g \quad \text{on} \quad \partial \Omega \times J~.
\label{eq:bc}
\end{aligned}
\end{equation}
Additionally, the initial condition is given by,
\begin{equation}
\begin{aligned}
p(x,0)=p^0(x). 
\label{eq:init}
\end{aligned}
\end{equation}
From now on our analysis focus on a transient (parabolic) problems which might be involved to various applications problems.   

\subsection{Transient problem with EVMFEM}
We start with the strong formulation of the transient (slightly compressible) flow problems governing single phase flow model for pressure $p$ and the velocity $\textbf{u}$, which is also case of slightly compressible single phase flow model: 
\begin{align} 
\textbf{u} &=-\mathbf{K} \nabla p \qquad \text{ in} \quad \Omega \times J,  \label{eq:spa} \\ 
\frac{\partial p}{\partial t}+\nabla \cdot \textbf{u} &=f \qquad \qquad \text{ in} \quad \Omega \times J, \label{eq:spb}  \\ 
p &=g \qquad \qquad \text{ on} \quad \partial \Omega \times J   \label{eq:spc} \\
p &=p_0 \qquad \qquad at \quad  t=0   \label{eq:spd}
\end{align}
where $\Omega \in \mathbb{R}^d( d=2$ or $3$) is multiblock domain, $J=[0, T]$  and $\mathbf{K}$ is a symmetric, uniformly positive definite tensor representing the permeability divided by the viscosity with $L^{\infty}(\Omega)$ components, for some $0<k_{min}<k_{max} < \infty$
\begin{align}
k_{min} \xi ^T \xi \le \xi^T \mathbf{K}(x) \xi \le k_{max} \xi^T \xi \qquad \forall x \in \Omega \quad \forall \xi \in \mathbb{R}^d, d=1,2,3.
\end{align}
The Dirichlet boundary condition is considered for convenience. A weak solution of parabolic Eqns. ($\ref{eq:spa}$) - ($\ref{eq:spd}$) is a pair $\{\mathbf{u}_h, p_h\} :J \rightarrow  \mathbf{V}^*_{h} \times W_h$, i.e. Enhanced Velocity Mixed Finite Element approximation
\begin{align} 
\left(K^{-1}\textbf{u}, \mathbf{v} \right) &=\left(p, \nabla \cdot \mathbf{v} \right)- \langle g, \mathbf{v} \cdot \nu \rangle_{\partial \Omega} \qquad & \forall \mathbf{v} \in \textbf{V} \label{eq:2_4p} \\ 
\left(\frac{\partial p}{\partial t},w\right)+\left(\nabla \cdot \textbf{u}, w \right) &=\left(f,w\right) \qquad \qquad &\forall w \in W \label{eq:2_5p}  
\end{align}
In addition, there is an initial condition
\begin{align} 
\left( p, w \right) \biggr \rvert_{t=0} &=\left(p_0,w\right) \qquad \qquad &\forall w \in W \label{eq:2_4pi} 
\end{align}
We formulate the variational problem in semi-discrete space as:
Find $\{\mathbf{u}_h, p_h\} :J \rightarrow  \mathbf{V}^*_{h} \times W_h$ such that
\begin{align} 
\left(\mathbf{K}^{-1}\textbf{u}_h, \mathbf{v} \right) &=\left(p_h, \nabla \cdot \mathbf{v} \right)- \langle g, \mathbf{v} \cdot \nu \rangle_{\partial \Omega} \quad & \forall \mathbf{v} \in \textbf{V}^*_h  \label{eq:2_6p} \\[10pt]
\left(\frac{\partial p_h}{\partial t},w\right)+\left(\nabla \cdot \textbf{u}_h, w \right) &=\left(f,w\right) \qquad \qquad &\forall w \in W_h\label{eq:2_7p}  
\end{align}
In addition, there is an initial condition
\begin{align} 
\left( p_h, w \right) \biggr \rvert_{t=0} &=\left(p_0,w\right) \qquad \qquad &\forall w \in W_h \label{eq:2_5pi} 
\end{align}
Subtracting Eqns. ($\ref{eq:2_4p}$) - ($\ref{eq:2_5p}$) from Eqns. ($\ref{eq:2_6p}$) - ($\ref{eq:2_7p}$) yields
\begin{align} 
\left(\mathbf{K}^{-1}(\textbf{u}-\textbf{u}_h), \mathbf{v} \right) -\left(p-p_h, \nabla \cdot \mathbf{v} \right)&=0 \qquad & \forall \mathbf{v} \in \textbf{V}^*_h  \label{eq:2_8} \\[10pt]
\left(\frac{\partial}{\partial t}(p-p_h),w\right)+\left(\nabla \cdot (\textbf{u}-\textbf{u}_h), w \right) &=0 \qquad \qquad &\forall w \in W_h   \label{eq:2_9}  
\end{align}
\newpage
\section{Error estimates}
In this section, we start with preliminaries including projections operators and some notations. Using them we present discrete formulations for analysis. Next, we derive auxiliary error estimates and a \textit{priori} error estimate theorems. 

\subsection{Projections}
We shall write a projection and define auxiliary error of pressure and velocity as follows:
\begin{align}
&E^I_p=p-\hat{p}, \; &E^A_p=\hat{p}-p_h , \\
&E^I_{\mathbf{u}}=\mathbf{u}-\Pi^*\mathbf{u}, \; &E^A_{\mathbf{u}}=\Pi^*\mathbf{u}-\mathbf{u}_h.
\end{align}
Note that  $\mathbf{u}-\mathbf{u}_h=E^I_{\mathbf{u}}+E^A_{\mathbf{u}}$, $p-p_h=E^I_p+E^A_p$. We used $\hat{p}$ the $L^2$-projection of $p$ that is defined as 
\begin{equation}\label{eqn:2_10}
\left(p-\hat{p}, w\right)=\left(E^I_p, w\right)=0 \qquad \forall w \in W_h. 
\end{equation}

We know from original work \cite{wheeler2002enhanced} that the projection operator $\Pi^*$ was introduced and was utilized for a \textit{priori} error analysis of elliptic problems. For convenience of the reader, we repeat the relevant and brief definition. Thus, we denote by $\Pi^*$ the projection operator that maps $(H^1(\Omega))^d$ onto $\mathbf{V}^*_h$ that defined locally for any element $T \in \mathcal{T}_h$ and any $\mathbf{q} \in (H^1(T))^d$ such that for all $\mathbf{q} \in (H^1(T))^d$ 
\begin{align}
\langle \Pi^* \mathbf{q} \cdot \nu, 1 \rangle_e = \langle \mathbf{q} \cdot \nu, 1 \rangle_e \label{eq:def_proj_star}
\end{align}
where $e$ is either any edge in 2D (or face in 3D) of $T$ not lying on $\Gamma$  or an edge in 2D (or face in 3D) of a sub-element, $T_k$. Such projection is developed prior to conducting error analysis for a \textit{priori} estimate.  As can be seen in Figure \ref{fig:EVMFEM_Gamma}, $T_k$ has a common edge with the interface grid $\mathcal{T}^{\Gamma}$. According to divergence theorem, we have
\begin{align}
\left(\nabla \cdot (\Pi^*\mathbf{q} -\mathbf{q}), w \right) = 0 \qquad \qquad \forall w \in W_h
\end{align} 

\begin{figure}[!tbp]
	\centering	
	\begin{tikzpicture}[thick,scale=1.5, dot/.style = {outer sep = +0pt, inner sep = +0pt, shape = circle, draw = black, label = {#1}},
	small dot/.style = {minimum size = 1pt, dot = {#1}},
	big dot/.style = {minimum size = 8pt, dot = {#1}},
	line join = round, line cap = round, >=triangle 45
	]
	
	\coordinate (A1) at (0, -1);
	\coordinate (A2) at (0, 1);
	\coordinate (A3) at (2, 1);
	\coordinate (A4) at (2, -1);
	
	\coordinate (B1) at (-2.6,-2);
	\coordinate (B2) at (-2.6, 1.5);
	\coordinate (B3) at (0,1.5);
	\coordinate (B4) at (0,-2);
	
	\coordinate (C1) at (-2.6,0);
	\coordinate (C2) at (0, 0);
	\coordinate (C3) at (2,0);
	
	\coordinate (D1) at (-2.6, -1);
	\coordinate (D2) at (-2.6,1);

	\fill[fill={rgb:orange,1;yellow,2;pink,5}, opacity=0.2] (A1) rectangle (A3); 
	
	\draw[very thick] (A1) -- (A2) -- (A3) -- (A4) -- cycle;
	
	\draw[fill={rgb:orange,1;yellow,2;blue,2},opacity=0.2, very thick] (B1) -- (B4)-- (C2)--(C1);
	\draw[fill={rgb:orange,2;yellow,1;green,1},opacity=0.2, very thick] (C1) -- (C2)-- (B3)--(B2);
	
	\draw[very thick] (B4) -- (B3);
	\draw[very thick] (B3) -- (B2);
	\draw[very thick] (C1) -- (C2);
	\draw[dotted, color =red] (C2) -- (C3);
	\draw[dotted, color =red] (A2) -- (D2);
	\draw[dotted, color =red] (A1) -- (D1);
	\draw[very thick] (B1) -- (B4);
	\draw[very thick] (B1) -- (B2);

	\node[font = \large, color =red] (e1)  at (0, 0.5) { \textbf{$\times$}};
	\node[font = \large, color =red] (e2)  at (0,-0.5) {$\times$};	
	\node[font = \large ] (e3)  at (1, -1) {$\times$};
	\node[font = \large] (e4)  at (2, 0) {$\times$};
	\node[font = \large] (e5)  at (1, 1) {$\times$};
	
	\node [above, font = \large] at (0,-2.7) { $\Gamma_{i,j}$};
	\node [above left, font = \large] at (0, 0.5) { $e_1$};
	\node [above left, font = \large] at (0,-0.5) { $e_2$};
	\node [right, font = \large] at (1, 0.5) { $T_1$};
	\node [right, font = \large] at (1,-0.5) { $T_2$};
	
	\end{tikzpicture}
	\caption{Degrees of freedom for the Enhanced Velocity space.}
	\label{fig:EVMFEM_Gamma}
\end{figure}
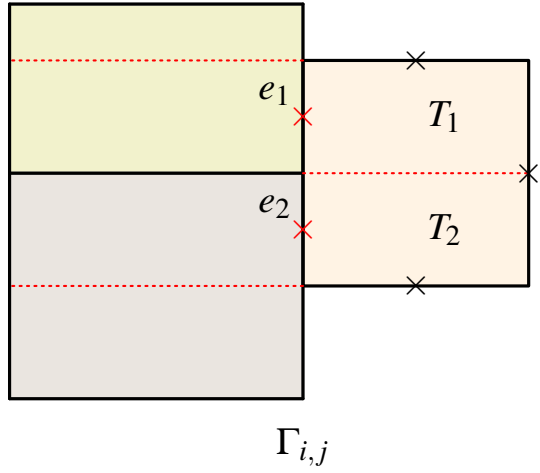

For $\mathbf{u} \in H^1(\Omega)$
\begin{equation} \label{eqn:2_11}
\left(\nabla \cdot \left(\Pi^*\mathbf{u}-\mathbf{u}\right), w\right)=\left(\nabla \cdot E^I_{\mathbf{u}}, w\right)=0 \qquad \forall w \in W_h 
\end{equation}

\textbf{Lemma.}
	\begin{equation}\label{eqn:2_12}
	\left(\frac{\partial}{\partial t} E^I_p, w\right)=0 \qquad \forall w \in W_h 
	\end{equation}
\begin{proof}
	$0=\frac{d}{dt}\left(E^I_p, w\right)=\left(\frac{\partial}{\partial t} E^I_p, w\right)+\cancelto{0}{\left( E^I_p, w_t\right)}$, 
	by \ref{eqn:2_10}, since $w_t \in W_h$. 
\end{proof}
Useful  inequalities of projections , see \cite{wheeler2002enhanced}:
\begin{equation} \label{eqn:2_13}
\norm{E^I_p} \le C \norm{p}_r h^r \qquad 0 \le r \le 1,
\end{equation}
\begin{equation} \label{eqn:2_14}
\norm{E^I_{\mathbf{u}}} \le C \norm{\mathbf{u}}_1 h.
\end{equation}
Recall Young's Inequality: for $a,b \ge 0$ 
\begin{equation} \label{eqn:youngineq}
ab \le \frac{1}{2\varepsilon} a^2+\frac{\varepsilon}{2}b^2
\end{equation}
The Inverse Inequality can be given as
\begin{equation}\label{eqn:inverseinequality}
\norm{\nabla \cdot \mathbf{u}_h} \le C h^{-1} \norm{\mathbf{u}_h}
\end{equation}
In this inequality, we have been working under the assumption that $\mathcal{T}_{h,i}$ quasi-uniform rectangular partition of $\Omega_i$. 

\subsection{Definitions}
In this section, we make analysis of discrete in time error estimates. Firstly, some definitions are made: for $\Delta t=\frac{T}{N}$, $N$ is a positive integer, $t_n=n \Delta t$ and for given $\theta \in [0,1]$,
\begin{equation}
f^n=f(x,t_n), \qquad 0 \le n \le N,
\end{equation}
\begin{equation}
f^{n,\theta}=\frac{1}{2}(1+\theta) f^{n+1}+\frac{1}{2}(1-\theta)f^n, \qquad 0 \le n \le N-1.
\end{equation}
Let's make also the following definitions:
\begin{align*}
&\norm{f}_{l^{\infty}(L^2)}=\underset{0 \leq n \leq N}{\max}  \norm{f^n}_{L^2} \\
&\norm{f}_{l^2(L^p)}=\left(\sum_{n=0}^{N-1} \norm{f^{n,\theta}}^2_{L^p} \Delta t \right)^{\frac{1}{2}}.
\end{align*}
We note that the difference can be expressed $t^{n,\theta}- t^{n}=\frac{1}{2}(1+\theta) \Delta t$. Next, using the Taylor series expansion about $t=t^{n,\theta}$, for any sufficiently smooth function $f(t)$, we obtain:
\begin{equation*}
\begin{split}
f^{n+1}=f \biggr \rvert_{t=t^{n,\theta}}+\frac{1}{2}(1-\theta) \Delta t \frac{\partial f}{\partial t} \biggr \rvert_{t=t^{n,\theta}}+\frac{1}{8}(1-\theta)^2 (\Delta t)^2 \frac{\partial^2 f}{\partial t^2} \biggr \rvert_{t=t^{n,\theta}}+ O(\Delta t^3)
\end{split}
\end{equation*}

\begin{equation*}
\begin{split}
f^{n}=f \biggr \rvert_{t=t^{n,\theta}}-\frac{1}{2}(1+\theta) \Delta t \frac{\partial f}{\partial t} \biggr \rvert_{t=t^{n,\theta}}+\frac{1}{8}(1+\theta)^2 (\Delta t)^2 \frac{\partial^2 f}{\partial t^2} \biggr \rvert_{t=t^{n,\theta}}+ O(\Delta t^3)
\end{split}
\end{equation*}

After multiplying the first equation by $\frac{1}{2}(1+\theta)$ and the second equation by $\frac{1}{2}(1-\theta)$ and then summing them, we obtain
\begin{equation*}
\begin{split}
f^{n,\theta}=f \biggr \rvert_{t=t^{n,\theta}}+\frac{1}{8} (\Delta t)^2 (1+\theta)(1-\theta)\frac{\partial^2 f}{\partial t^2} \biggr \rvert_{t=t^{n,\theta}}+ O(\Delta t^3)
\end{split}
\end{equation*}
Note that if $\theta =1 $ then $f^{n, \theta}=f \biggr \rvert_{t=t^{n,\theta}} + O(\Delta t^3)$. In addition, we can get second order approximation of $\Delta t$, details in \cite{riviere2000discontinuous} : 
$ p(\mathbf{x},t^{n,\theta}) \approx p^{n, \theta} \text{  and  } u(\mathbf{x},t^{n,\theta}) \approx u^{n, \theta}.$
According to Taylor series expansion \cite{riviere2000discontinuous}, we obtain
\begin{equation} \label{eqn:3_3}
\frac{p^{n+1}-p^n}{\Delta t}=p_t(x,t^{n,\theta})+\rho^{p,n,\theta}, \qquad \forall x \in \Omega,
\end{equation}
where $\rho^{p,n,\theta}$ depends on time-derivatives of $p$ and $\Delta t$
\begin{equation} \label{eqn:3_3a}
\norm{\rho^{p,n,\theta}} \le
\begin{cases*}
C_1 \Delta t  \norm{p_{tt}}_{L^{\infty}((t^n,t^{n+1}),H^1)}, \qquad  if \quad \theta =1, \\
C_2 \Delta t^2 \norm{p_{ttt}}_{L^{\infty}((t^n,t^{n+1}),H^1)} , \qquad  if \quad \theta =0,
\end{cases*}
\end{equation}
so $\norm{\rho^{p,n,\theta}}=\mathcal{O}(\Delta t^{2-\theta})$.
\subsection{Discrete formulation}
We formulate variational form in semi-discrete space as: Find $\{\mathbf{u}_h, p_h\} :J \rightarrow  \mathbf{V}^*_{h} \times W_h$ such that
\begin{align} 
\left(\frac{\partial p_h}{\partial t},w\right)+\left(\nabla \cdot \textbf{u}_h, w \right) &=l_{1}(w)  \qquad \qquad &\forall w \in W_h \label{eq:3_6} \\[10pt]
\left(\mathbf{K}^{-1}\textbf{u}_h, \mathbf{v} \right) -\left(p_h, \nabla \cdot \mathbf{v} \right)&=l_{2}(\mathbf{v})\quad & \forall \mathbf{v} \in \textbf{V}^*_h \label{eq:3_7}    
\end{align}

In addition, there is an initial condition
\begin{align} 
\left( p_h, w \right) \biggr \rvert_{t=0} &=\left(p_0,w\right) \qquad \qquad &\forall w \in W_h \label{eq:3_5pi} 
\end{align}

where $l_1$ and $l_2$ are bounded linear functionals, i.e.
\begin{equation*}
\begin{split}
&l_{1}(w)=\left(f,w\right), \\
&l_{2}(\mathbf{v})=- \langle g, \mathbf{v} \cdot \nu \rangle_{\partial \Omega}.
\end{split}
\end{equation*}

With this definitions, Eqns. ($\ref{eq:3_6}$) - ($\ref{eq:3_7}$) become as: Find  $\{\mathbf{u}^{n, \theta}_h, p^{n, \theta}_h\} \in \mathbf{V}^*_{h} \times W_h$, $n=1,2, ..., N-1$, such that
\begin{align} 
&\left(\frac{p^{n+1}_h- p^{n}_h}{\Delta t},w\right)+\left(\nabla \cdot \textbf{u}^{n, \theta}_h, w \right) =l^{n, \theta}_{1}(w) \qquad \qquad &\forall w \in W_h\label{eq:3_6p}  \\[10pt]
&\left(\mathbf{K}^{-1}\textbf{u}^{n, \theta}_h, \mathbf{v} \right) -\left(p^{n, \theta}_h, \nabla \cdot \mathbf{v} \right)=l^{n, \theta}_{2}(\mathbf{v})  \quad  &\forall \mathbf{v} \in \textbf{V}^*_h  \label{eq:3_7p} 
\end{align}
Note that if $\theta =1$ then the time discretization is the backward Euler(Implicit) method, and if $\theta=0$ then the Crank-Nicolson scheme.

We consider true solution $\mathbf{u} \in L^2\left(J,\mathbf{V}\right) $ and $p_h \in H^1\left(J,W\right)$ of Eqns. ($\ref{eq:2_4p}$) and ($\ref{eq:2_5p}$) at time $t=t^{n, \theta}$ in the continuous in time with spatially discrete scheme. We used Eqn. ($\ref{eqn:3_3}$) and additional remark related to the Taylor series expansion in order to obtain the following equations with at least order of $\mathcal{O}(\Delta t)$:
\begin{align} 
&\left(\frac{ p^{n+1}-p^{n}}{\Delta t},w\right)+\left(\nabla \cdot \textbf{u}^{n,\theta}, w \right) =l_1(w)+\left( \rho^{p,n,\theta},w\right) \qquad \forall w \in W \label{eq:3_8} \\
&\left(K^{-1}\textbf{u}^{n,\theta}, \mathbf{v} \right) -\left(p^{n,\theta}, \nabla \cdot \mathbf{v} \right)=l_2(\mathbf{v})\qquad \qquad  \forall \mathbf{v} \in \textbf{V} \label{eq:3_9}   
\end{align}

We approximate the vector integrals type $\left( \mathbf{v}, \mathbf{q} \right)_{T, M}$ by trapezoidal-midpoint quadrature rules and $\left(\mathbf{K}^{-1}\mathbf{q}, \mathbf{v}\right)_{T}$ by trapezoidal  quadrature rules respectively. In \cite{russell1983finite}, the equivalence between finite volume methods and the mixed finite element method was established for special quadrature rule for $\mathbf{K}$ diagonal tensor and using the lowest-order Raviart Thomas spaces on rectangles. We emphasize that the EVMFEM with special quadrature and velocity elimination in the discrete system can be reduced to well-known a cell-centered finite difference method. 

For each time step we use the Newton method to solve the system, in case of the slightly compressible flow: the fluid compressibility $C_{f}$ term brings us to a nonlinear system. That is why consideration of it would be beneficial for nonlinear problems in the future.

\subsection{Analysis}
We first derive the bounds of auxiliary error terms.

\begin{thm}[Auxiliary error estimate] \label{thm:3_1}
	For the velocity $\mathbf{u_h}$ and pressure $p_h$ of the mixed method spaces $\textbf{V}^*_h \times W_h $ satisfying equations ($\ref{eq:3_6p}$) - ($\ref{eq:3_7p}$), assume $\Delta t$ is sufficiently small and positive, $\mathbf{K}$ is uniformly positive definite and sufficient regularity of true solution in equations (\ref{eq:spa})-(\ref{eq:spd}). Then, there exist a constant $C$ such that
	\begin{equation}
	\norm{E_{\mathbf{u}}^{A}}^2_{l^{2}(L^2)} + \norm{E_p^{A}}^2_{l^{\infty}(L^2)} \le C \left( h^2+h+\Delta t^{2r} \right)
	\end{equation}
	where $C=C(T, \mathbf{K}, \mathbf{u}, p)$ and 
	\begin{equation*}
	r=
	\begin{cases*}
	1, \qquad  if \quad \theta =1 \\
	2, \qquad  if \quad \theta =0
	\end{cases*}
	\end{equation*}
\end{thm}

\begin{proof}

Subtracting Eqns. ($\ref{eq:3_8})-(\ref{eq:3_9}$) from Eqns. ($\ref{eq:3_6p})-(\ref{eq:3_7p}$) respectively yields

\begin{align} 
&\left(\frac{p^{n+1}-p^{n+1}_h- \left(p^{n}-p^{n}_h\right)}{\Delta t},w\right)+\left(\nabla \cdot \left(\textbf{u}^{n, \theta}-\textbf{u}^{n, \theta}_h\right), w \right) =\left( \rho^{p,n,\theta},w\right) \qquad \qquad \forall w \in W_h\label{eq:3_10}  \\[1 em]
&\left(K^{-1}\left(\textbf{u}^{n, \theta}-\textbf{u}^{n, \theta}_h\right), \mathbf{v} \right) -\left(p^{n, \theta}-p^{n, \theta}_h, \nabla \cdot \mathbf{v} \right)=0  \quad  \forall \mathbf{v} \in \textbf{V}^*_h  \label{eq:3_11}. 
\end{align}
Take $\mathbf{v}=\Pi^*\mathbf{u}^{n,\theta}-\mathbf{u}^{n,\theta}_h=E_{\mathbf{u}}^{A \; n, \theta}$ and $w=E_p^{A \; n, \theta}$ in ($\ref{eq:3_11}$) and ($\ref{eq:3_10}$) respectively. 
\begin{equation*}
\begin{split}
&\left(\frac{\left(E_p^{I \; n+1}+E_p^{A \; n+1} \right)- \left( E_p^{I \; n}+E_p^{A \; n} \right)}{\Delta t},E_p^{I \; n, \theta}\right) \\[10 pt]
&+\left(\nabla \cdot \left(E_{\mathbf{u}}^{I \; n, \theta}+E_{\mathbf{u}}^{A \; n, \theta}\right), E_p^{I \; n, \theta} \right) =\left( \rho^{p,n,\theta},E_p^{A \; n, \theta}\right)
\end{split}
\end{equation*}
\begin{equation*}
\begin{split}
\left(\mathbf{K}^{-1}\left(E_{\mathbf{u}}^{I \; n, \theta}+E_{\mathbf{u}}^{A \; n, \theta}\right), E_{\mathbf{u}}^{A \; n, \theta} \right) -\left(E_p^{I \; n, \theta}+E_p^{A \; n, \theta}, \nabla \cdot E_{\mathbf{u}}^{A \; n, \theta} \right)=0
\end{split}
\end{equation*}
After adding them, we can rewrite as
\begin{align*}
&\left(\frac{\left(E_p^{I \; n+1}+E_p^{A \; n+1} \right)- \left( E_p^{I \; n}+E_p^{A \; n} \right)}{\Delta t},E_p^{A \; n, \theta}\right)+\left(\nabla \cdot \left(E_{\mathbf{u}}^{I \; n, \theta}+E_{\mathbf{u}}^{A \; n, \theta}\right), E_p^{A \; n, \theta} \right) +\\[10pt]
&+\left(\mathbf{K}^{-1}\left(E_{\mathbf{u}}^{I \; n, \theta}+E_{\mathbf{u}}^{A \; n, \theta}\right), E_{\mathbf{u}}^{A \; n, \theta} \right) -\left(E_p^{I \; n, \theta}+E_p^{A \; n, \theta}, \nabla \cdot E_{\mathbf{u}}^{A \; n, \theta} \right)=\\[10pt]
&=\left(\frac{E_p^{I \; n+1}+E_p^{A \; n+1} }{\Delta t},E_p^{A \; n}\right)-  \left(\frac{  E_p^{I \; n}+E_p^{A \; n} }{\Delta t},E_p^{A \; n, \theta}\right)\\[10pt]
&+\left(\nabla \cdot E_{\mathbf{u}}^{I \; n, \theta}, E_p^{A \; n, \theta} \right) 
+\cancelto{0}{\left(\nabla \cdot E_{\mathbf{u}}^{A \; n, \theta}, E_p^{A \; n, \theta} \right)}\\[10pt]
&+\left(\mathbf{K}^{-1}\left(E_{\mathbf{u}}^{I \; n, \theta}+E_{\mathbf{u}}^{A \; n, \theta}\right), E_{\mathbf{u}}^{A \; n, \theta} \right) -\left(E_p^{I \; n, \theta}, \nabla \cdot E_{\mathbf{u}}^{A \; n, \theta} \right)
-\cancelto{0}{\left(E_p^{A \; n, \theta}, \nabla \cdot E_{\mathbf{u}}^{A \; n, \theta} \right)} \\[10pt]
&=\cancelto{0, \; by \; \ref{eqn:2_10}}{\left(\frac{E_p^{I \; n+1} }{\Delta t},E_p^{A \; n, \theta}\right)}\qquad + \qquad
\left(\frac{E_p^{A \; n+1} }{\Delta t},E_p^{A \; n, \theta}\right) \\[10pt]
&-\cancelto{0, \; by \; \ref{eqn:2_10}}{\left(\frac{  E_p^{I \; n} }{\Delta t},E_p^{A \; n, \theta}\right)}
-  \left(\frac{ E_p^{A \; n} }{\Delta t},E_p^{A \; n, \theta}\right) 
+\cancelto{0, \; by \; \ref{eqn:2_11}}{\left(\nabla \cdot E_{\mathbf{u}}^{I \; n, \theta}, E_p^{A \; n, \theta} \right)} \\[10pt] 
&+\left(\mathbf{K}^{-1}\left(E_{\mathbf{u}}^{I \; n, \theta}+E_{\mathbf{u}}^{A \; n, \theta}\right), E_{\mathbf{u}}^{A \; n, \theta} \right) -\left(E_p^{I \; n, \theta}, \nabla \cdot E_{\mathbf{u}}^{A \; n, \theta} \right)\\[13pt]
&=\left(\frac{E_p^{A \; n+1}-E_p^{A \; n} }{\Delta t},E_p^{A \; n, \theta}\right)
+\left(\mathbf{K}^{-1}\left(E_{\mathbf{u}}^{I \; n, \theta}+E_{\mathbf{u}}^{A \; n, \theta}\right), E_{\mathbf{u}}^{A \; n, \theta} \right) -\left(E_p^{I \; n, \theta}, \nabla \cdot E_{\mathbf{u}}^{A \; n, \theta} \right)\\[10pt]
&=\left(\rho^{p,n,\theta},E_p^{A \; n, \theta}\right)
\end{align*}

\begin{equation}\label{eq:3_12}
\begin{split}
&\left(\frac{E_p^{A \; n+1}-E_p^{A \; n}  }{\Delta t},E_p^{A \; n, \theta}\right)
+ \left(\mathbf{K}^{-1} E_{\mathbf{u}}^{A \; n, \theta}, E_{\mathbf{u}}^{A \; n, \theta} \right) \\[15pt] 
&=-\left(\mathbf{K}^{-1} E_{\mathbf{u}}^{I \; n, \theta}, E_{\mathbf{u}}^{A \; n, \theta} \right)+
\left(E_p^{I \; n, \theta}, \nabla \cdot E_{\mathbf{u}}^{A \; n, \theta} \right)+
\left( \rho^{p,n,\theta},E_p^{A \; n, \theta}\right)
\end{split}
\end{equation}

\begin{align*}
&\left(\frac{E_p^{A \; n+1}-E_p^{A \; n}  }{\Delta t},E_p^{A \; n, \theta}\right)=\left(\frac{E_p^{A \; n+1}-E_p^{A \; n}  }{\Delta t},\frac{1+\theta}{2}E_p^{A \; n+1, \theta}+\frac{1-\theta}{2}E_p^{A \; n, \theta}\right) \\[10pt]
&=\frac{1+\theta}{2 \Delta t}\left(E_p^{A \; n+1}-E_p^{A \; n} ,E_p^{A \; n+1}\right)+
\frac{1-\theta}{2 \Delta t}\left(E_p^{A \; n+1}-E_p^{A \; n} ,E_p^{A \; n}\right) \\[10pt]
&=\frac{1+\theta}{2 \Delta t}\left(E_p^{A \; n+1} ,E_p^{A \; n+1}\right)-
\frac{1+\theta}{2 \Delta t}\left(E_p^{A \; n} ,E_p^{A \; n+1}\right)\\[10pt]
&+\frac{1-\theta}{2 \Delta t}\left(E_p^{A \; n+1} ,E_p^{A \; n}\right) -\frac{1-\theta}{2 \Delta t}\left(E_p^{A \; n} ,E_p^{A \; n}\right) \\[10pt]
&=\frac{1+\theta}{2 \Delta t}\norm{E_p^{A \; n+1}}^2-
\frac{2\theta}{2\Delta t}\left(E_p^{A \; n} ,E_p^{A \; n+1}\right)-
\frac{1-\theta}{2 \Delta t}\norm{E_p^{A \; n}}^2 \\[10pt]
&=\frac{1}{2 \Delta t}\norm{E_p^{A \; n+1}}^2-\frac{1}{2 \Delta t}\norm{E_p^{A \; n}}^2
+\frac{\theta}{2 \Delta t}\left(\norm{E_p^{A \; n+1}}^2-2\left(E_p^{A \; n} ,E_p^{A \; n+1}\right)+\norm{E_p^{A \; n}}^2\right) \\[10pt]
&=\frac{1}{2 \Delta t}\left(\norm{E_p^{A \; n+1}}^2-\norm{E_p^{A \; n}}^2\right)
+\underbrace{\frac{\theta}{2 \Delta t}\left(\norm{E_p^{A \; n+1}}-\norm{E_p^{A \; n}}\right)^2}_{\ge 0} \ge \\[10pt]
&\ge \frac{1}{2 \Delta t}\left(\norm{E_p^{A \; n+1}}^2-\norm{E_p^{A \; n}}^2\right)  	
\end{align*}
It follows immediately that 
\begin{equation} \label{eq:3_13}
\left(\frac{E_p^{A \; n+1}-E_p^{A \; n}  }{\Delta t},E_p^{A \; n, \theta}\right)
\ge \frac{1}{2 \Delta t}\left(\norm{E_p^{A \; n+1}}^2-\norm{E_p^{A \; n}}^2\right)  	
\end{equation}
By using $\ref{eq:3_13}$, multiply by $2 \Delta t$ and sum from $0$ to $N-1$ in Equation $\ref{eq:3_12}$.

\begin{equation*}
\begin{split}
&\sum^{N-1}_{n=0} \left(\norm{E_p^{A \; n+1}}^2-\norm{E_p^{A \; n}}^2\right) 
+ 2\sum^{N-1}_{n=0} \left(\mathbf{K}^{-1} E_\mathbf{u}^{A \; n, \theta}, E_\mathbf{u}^{A \; n, \theta} \right) \Delta t \le \\[10pt]
&\le -2\sum^{N-1}_{n=0} \left(\mathbf{K}^{-1} E_{\mathbf{u}}^{I \; n, \theta}, E_{\mathbf{u}}^{A \; n, \theta} \right) \Delta t+
2\sum^{N-1}_{n=0} \left(E_p^{I \; n, \theta}, \nabla \cdot E_{\mathbf{u}}^{A \; n, \theta} \right) \Delta t+
2\sum^{N-1}_{n=0} \left(\rho^{p,n,\theta},E_p^{A \; n, \theta}\right) \Delta t
\end{split}
\end{equation*}

\begin{equation*}
\begin{split}
&\left(\norm{E_p^{A \; N}}^2-\cancelto{0, \; by \; \ref{eq:3_5pi}}{\norm{E_p^{A \; 0}}^2}\right) 
\qquad  \quad + \sum^{N-1}_{n=0} \left(\mathbf{K}^{-1} E_{\mathbf{u}}^{A \; n, \theta}, E_{\mathbf{u}}^{A \; n, \theta} \right) \Delta t  \\[15pt]
&\le \underbrace{-2\sum^{N-1}_{n=0} \left(\mathbf{K}^{-1} E_{\mathbf{u}}^{I \; n, \theta}, E_{\mathbf{u}}^{A \; n, \theta} \right) \Delta t}_{\mathbb{T}_1}+
\underbrace{2\sum^{N-1}_{n=0} \left( \rho^{p,n,\theta},E_p^{A \; n, \theta}\right) \Delta t }_{\mathbb{T}_2}+
\underbrace{2\sum^{N-1}_{n=0} \left(E_p^{I \; n, \theta}, \nabla \cdot E_{\mathbf{u}}^{A \; n, \theta} \right) \Delta t}_{\mathbb{T}_3}
\end{split}
\end{equation*} 
\begin{align*}
\mathbb{T}_1 &=-2\sum^{N-1}_{n=0} \left(\mathbf{K}^{-1} E_{\mathbf{u}}^{I \; n, \theta}, E_{\mathbf{u}}^{A \; n, \theta} \right) \Delta t \\[10pt] &\underbrace{\le}_{Holder's \; ineq.} 2\sum^{N-1}_{n=0} \norm{\mathbf{K}^{-1} E_{\mathbf{u}}^{I \; n, \theta}} \norm{E_{\mathbf{u}}^{A \; n, \theta}} \Delta t \\[10pt]
&\underbrace{\le}_{Young's \; ineq.} \frac{1}{\varepsilon k^2_{min} }\sum^{N-1}_{n=0} \norm{ E_{\mathbf{u}}^{I \; n, \theta}}^2 \Delta t + \varepsilon \sum^{N-1}_{n=0} \norm{E_{\mathbf{u}}^{A \; n, \theta}}^2 \Delta t
\end{align*}
We use the Holder inequality and the Young inequality to get
\begin{align*}
\mathbb{T}_2=&2\sum^{N-1}_{n=0} \left( \rho^{p,n,\theta},E_p^{A \; n, \theta}\right)\Delta t \le 2\sum^{N-1}_{n=0} \norm{ \rho^{p,n,\theta}}\norm{E_p^{A \; n, \theta} } \Delta t   \\[10pt]
&\le \sum^{N-1}_{n=0} \norm{E_p^{A \; n }}^2 \Delta t + \sum^{N-1}_{n=0} \norm{\rho^{p,n,\theta}}^2 \Delta t.   
\end{align*}
We should note that 
\begin{equation}\label{eqn:2_20a}
\begin{split}
\left(E^I_p, \nabla \cdot E^A_u \right)_{\Omega}=\left(E^I_p, \nabla \cdot E^A_u \right)_{\Omega^*}+\cancelto{0}{\left(E^I_p, \nabla \cdot E^A_u \right)_{\Omega \backslash \Omega^*}}=\left(E^I_p, \nabla \cdot E^A_u \right)_{\Omega^*}
\end{split}
\end{equation}
since $\nabla \cdot E^A_u  \biggr \rvert_{\Omega \backslash \Omega^*} \in W_h$ and the property $\ref{eqn:2_10}$.
\begin{align*}
\mathbb{T}_3 &=2\sum^{N-1}_{n=0} \left(E_p^{I \; n, \theta}, \nabla \cdot E_{\mathbf{u}}^{A \; n, \theta} \right)_{\Omega} \Delta t \\[10pt] &=2\sum^{N-1}_{n=0} \left(E_p^{I \; n, \theta}, \nabla \cdot E_{\mathbf{u}}^{A \; n, \theta} \right)_{\Omega^*} \Delta t  \\[10pt] 
&\le 2\sum^{N-1}_{n=0} \norm{E_p^{I \; n, \theta}}_{\Omega^*} \norm{\nabla \cdot E_{\mathbf{u}}^{A \; n, \theta}}_{\Omega^*} \Delta t \le \\[10pt] 
&\le 2C \sum^{N-1}_{n=0} \left(\frac{1+\theta}{2}\norm{p^{n+1}}_{1, \Omega^*}+\frac{1-\theta}{2}\norm{p^{n}}_{1, \Omega^*} \right)h \norm{ E_{\mathbf{u}}^{A \; n, \theta}}_{\Omega} h^{-1} \Delta t \\[10pt]
&\le C \sum^{N-1}_{n=0} \left(\frac{1+\theta}{2}\norm{p^{n+1}}_{1, \Omega^*}+\frac{1-\theta}{2}\norm{p^{n}}_{1, \Omega^*} \right)^2 \Delta t + \varepsilon \sum^{N-1}_{n=0} \norm{ E_{\mathbf{u}}^{A \; n, \theta}}^2  \Delta t  
\end{align*}
\begin{remark}
	We used the following properties:
	\begin{align*}
	\norm{E_p^{I \; n, \theta}}_{\Omega^*}&=\norm{\frac{1+\theta}{2}E_p^{I \; n+1}+\frac{1-\theta}{2}E_p^{I \; n}}_{\Omega^*} \\[10pt]  
	&\le \frac{1+\theta}{2}\norm{E_p^{I \; n+1}}_{\Omega^*}+\frac{1-\theta}{2}\norm{E_p^{I \; n}}_{\Omega^*}  \\[10pt] 
	&\le \frac{1+\theta}{2}\norm{p^{n+1}}_{1, \Omega^*} h+\frac{1-\theta}{2}\norm{p^{n}}_{1, \Omega^*} h \\[10pt] 
	&\le \left(\frac{1+\theta}{2}\norm{p^{n+1}}_{1, \Omega^*}+\frac{1-\theta}{2}\norm{p^{n}}_{1, \Omega^*} \right)h
	\end{align*}
	and
	\begin{align*}
	\norm{\nabla \cdot E_{\mathbf{u}}^{A \; n, \theta}}_{\Omega^*} \le C \norm{ E_{\mathbf{u}}^{A \; n, \theta}}_{\Omega^*} h^{-1}
	\le C \norm{ E_{\mathbf{u}}^{A \; n, \theta}}_{\Omega} h^{-1}
	\end{align*}
\end{remark}
Next, we know that
\begin{equation*}
\begin{split}
\left(\mathbf{K}^{-1} E_{\mathbf{u}}^{A \; n, \theta}, E_{\mathbf{u}}^{A \; n, \theta} \right) \ge \frac{1}{k_{max}} \norm{E_{\mathbf{u}}^{A \; n, \theta} }^2
\end{split}
\end{equation*}

Therefore,
\begin{equation*}
\begin{split}
&\frac{1}{2} \norm{E_p^{A \; N}}^2+\left[\frac{2}{k_{max}}-2\varepsilon \right]  \sum^{N-1}_{n=0} \norm{E_{\mathbf{u}}^{A \; n, \theta} }^2 \Delta t  \le \\[10pt] 
&\le C \sum^{N-1}_{n=0} \norm{ E_{\mathbf{u}}^{I \; n, \theta}}^2 \Delta t+
C\sum^{N-1}_{n=0} \norm{\rho^{p,n,\theta}}^2 \Delta t  + C \sum^{N-1}_{n=0} \norm{E_p^{A \; n }}^2 \Delta t+ \\[10pt] 
&+\frac{1}{2} \norm{ E_p^{A \; N}}^2 \Delta t +
+ C \sum^{N-1}_{n=0} \left(\frac{1+\theta}{2}\norm{p^{n+1}}_{1, \Omega^*}+\frac{1-\theta}{2}\norm{p^{n}}_{1, \Omega^*} \right)^2 \Delta t 
\end{split}
\end{equation*}

We can multiply by 2 and make $\varepsilon$ small enough in order to have LHS with positive coefficients. Later take minimum and divide both side of inequality.

\begin{align*}
&\sum^{N-1}_{n=0} \norm{E_{\mathbf{u}}^{A \; n, \theta} }^2 \Delta t +\norm{E_p^{A \; N}}^2 \le \\[10pt] 
&\le  C \Delta t \left[ \sum^{N-1}_{n=0} \norm{ E_{\mathbf{u}}^{I \; n, \theta}}^2+\sum^{N-1}_{n=0} \left(\frac{1+\theta}{2}\norm{p^{n+1}}_{1, \Omega^*}+\frac{1-\theta}{2}\norm{p^{n}}_{1, \Omega^*} \right)^2  \right] \\[10pt] 
&+ C \sum^{N-1}_{n=0} \norm{E_p^{A \; n} }^2 \Delta t  + C\sum^{N-1}_{n=0} \norm{\rho^{p,n,\theta}}^2 \Delta t. 
\end{align*}
We are thus apply the discrete Gronwall lemma, for sufficiently small $\Delta t$, to obtain:
\begin{equation*}
\begin{split}
&\sum^{N-1}_{n=0} \norm{E_{\mathbf{u}}^{A \; n, \theta} }^2 \Delta t +\norm{E_p^{A \; N}}^2 \\[10pt] 
&\le C \Delta t \left[ \sum^{N-1}_{n=0} \norm{ E_{\mathbf{u}}^{I \; n, \theta}}^2+\sum^{N-1}_{n=0} \left(\frac{1+\theta}{2}\norm{p^{n+1}}_{1, \Omega^*}+\frac{1-\theta}{2}\norm{p^{n}}_{1, \Omega^*}\right)^2  \right] + C\sum^{N-1}_{n=0} \norm{\rho^{p,n,\theta}}^2 \Delta t  \\[10pt] 
&\le  C \Delta t \left[ \sum^{N-1}_{n=0} \norm{ E_{\mathbf{u}}^{I \; n, \theta}}^2+h \sum^{N-1}_{n=0} \left(\frac{1+\theta}{2}\norm{p^{n+1}}_{1,\infty, \Omega^*}+\frac{1-\theta}{2}\norm{p^{n}}_{1,\infty, \Omega^*} \right)^2  \right] + C\sum^{N-1}_{n=0} \norm{\rho^{p,n,\theta}}^2 \Delta t  \\[10pt] 
&\le  C h^2  \sum^{N-1}_{n=0} \left(\frac{1+\theta}{2}\norm{u^{n+1}}_{1}+\frac{1-\theta}{2}\norm{u^{n}}_{1} \right)^2 \Delta t+\\[10pt] 
&+C  h \sum^{N-1}_{n=0}\left(\frac{1+\theta}{2}\norm{p^{n+1}}_{1,\infty, \Omega^*}+\frac{1-\theta}{2}\norm{p^{n}}_{1,\infty, \Omega^*} \right)^2 \Delta t  + C\sum^{N-1}_{n=0} \norm{\rho^{p,n,\theta}}^2 \Delta t  \\[10pt] 
&\le C(T,\mathbf{u},p,\mathbf{K})\left(h^2+h+\Delta t^{2r}\right).
\end{split}
\end{equation*}
where 
\begin{equation*}
r=
\begin{cases*}
1, \qquad  if \quad \theta =1 \\
2, \qquad  if \quad \theta =0
\end{cases*}
\end{equation*}
We used the fact that $\sum^{N-1}_{n=0} \Delta t g_n \le C T \sum^{N-1}_{n=0} g_n $ , $|\Omega^*| \le Ch$ and property that is given in eqn. ($\ref{eqn:3_3a}$). This finishes the proof of theorem.
\end{proof}

The auxiliary error estimates theorem allows us to conclude the following theorem: 

\begin{thm}[Error estimate]
	Assume the same conditions as in the previous theorem. Then, 
	\begin{equation}
	\norm{p-p_h}^2_{l^{\infty}(L^2)}+\norm{\mathbf{u}-\mathbf{u}_h}^2_{l^2(L^2)} \le C \left( h^2+h+\Delta t^{2r}  \right)
	\end{equation}
	where $C=C(T, \mathbf{K}, \mathbf{u}, p)$ and 
	\begin{equation*}
	r=
	\begin{cases*}
	1, \qquad  if \quad \theta =1 \\
	2, \qquad  if \quad \theta =0
	\end{cases*}
	\end{equation*}
	\begin{proof}
		By applying triangle inequality,the Interpolation Error Inequalities and Theorem $\ref{thm:3_1}$ results we obtain:
		\begin{equation*}
		\begin{split}
		&\norm{p-p_h}^2_{l^{\infty}(L^2)}+\norm{\mathbf{u}-\mathbf{u}_h}^2_{l^2(L^2)} = \norm{E^I_p+E^A_p}^2_{l^{\infty}(L^2)}+\norm{E^I_{\mathbf{u}}+E^A_{\mathbf{u}}}^2_{l^2(L^2)} \le \\[10pt]
		&\le C \left(\underbrace{\norm{E^I_p}^2_{l^{\infty}(L^2)}+\norm{E^I_{\mathbf{u}}}^2_{l^2(L^2)}}_{Interpolation \; error}+\underbrace{\norm{E^A_p}^2_{l^{\infty}(L^2)}+\norm{E^A_{\mathbf{u}}}^2_{l^2(L^2)}}_{Auxiliary \; error}\right)   \le \\[10pt]
		&\le C(T,p,\mathbf{u},\mathbf{K}) (h^2+h)+O(\Delta t^{2r})
		\end{split}
		\end{equation*}
	\end{proof}
\end{thm}

\newpage

\section{Numerical Examples}
\index{Numerical examples%
	@\emph{Numerical examples}}%
In this section, we conduct numerical experiment to verify the numerical accuracy of parabolic problem solution using EVMFEM in space and the backward Euler in time. Based on our a \textit{priori} error analysis estimates we assume a sufficiently smooth analytical solution. In numerical examples, we set $\Omega = \left(0, 1\right) \times \left(0, 1\right) $, $\mathbf{K}_{i,j}=\delta_{i,j}$ and the domain $\Omega$ is divided into four subdomains $\Omega_i$; $\Omega_1$ and $\Omega_4$ have fine grids, $\Omega_2$ and $\Omega_3$ have coarse grids;   such mesh discretization is illustrated in Figure \ref{fig:mesh_discretization}.

\subsection{Numerical example 1}
We use the known solution
\begin{align*}
p(x,y,t)=tx(1-x)y(1-y)
\end{align*}
and use it to compute the forcing $f$, the Dirichlet boundary data $g$, and the initial data $p_0$. We carry out several levels of uniform grid refinement in each subdomains. The time step and the element size are almost equal to each other, see Table \ref{tab:errortimeN1}.
The simulation time interval is $(0; 0.1)$, i.e. $T=0.1$, and we use the Backward Euler method to integrate with regard to time with uniform time step. 
We are interested in finding the exact error using a given true solution, so the pressure is true error and the velocity error is normalized error. On applying sufficient Newton iterations at each time step provided the residual is within the machine-precision tolerance, we obtain the numerical solution for evaluating of the error in specified norm. We compute $error_p$, corresponds to $\norm{p-p_h}_{l^{\infty}(L^2)} $, which is maximum of values among time steps that resulting for given time step a discrete pressure $L^2$-norm that associates only the function values at the cell-centers in space. Also, $error_u$ is defined as $\norm{\mathbf{u}-\mathbf{u}_h}_{l^2(L^2)}$ where in space a discrete $L^2$-norm that associates only the normal vector components at the midpoint edges and then  normalized by $\norm{\mathbf{u}}_{L^2}$ and $l^2$-norm in time. The convergence rate is illustrated in Figure \ref{fig:testN4_error_pres}.
\begin{figure}[H]
	\centering
	\includegraphics[width=0.3\linewidth]{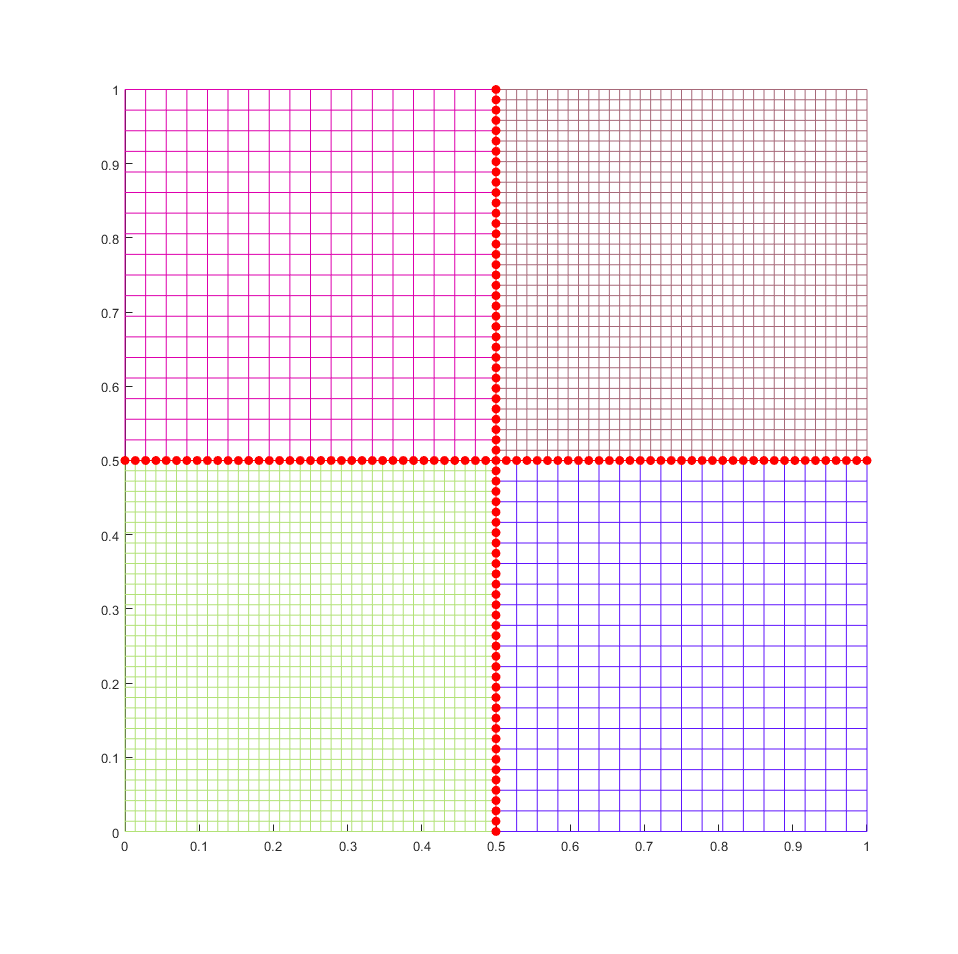}
	\includegraphics[width=0.3\linewidth]{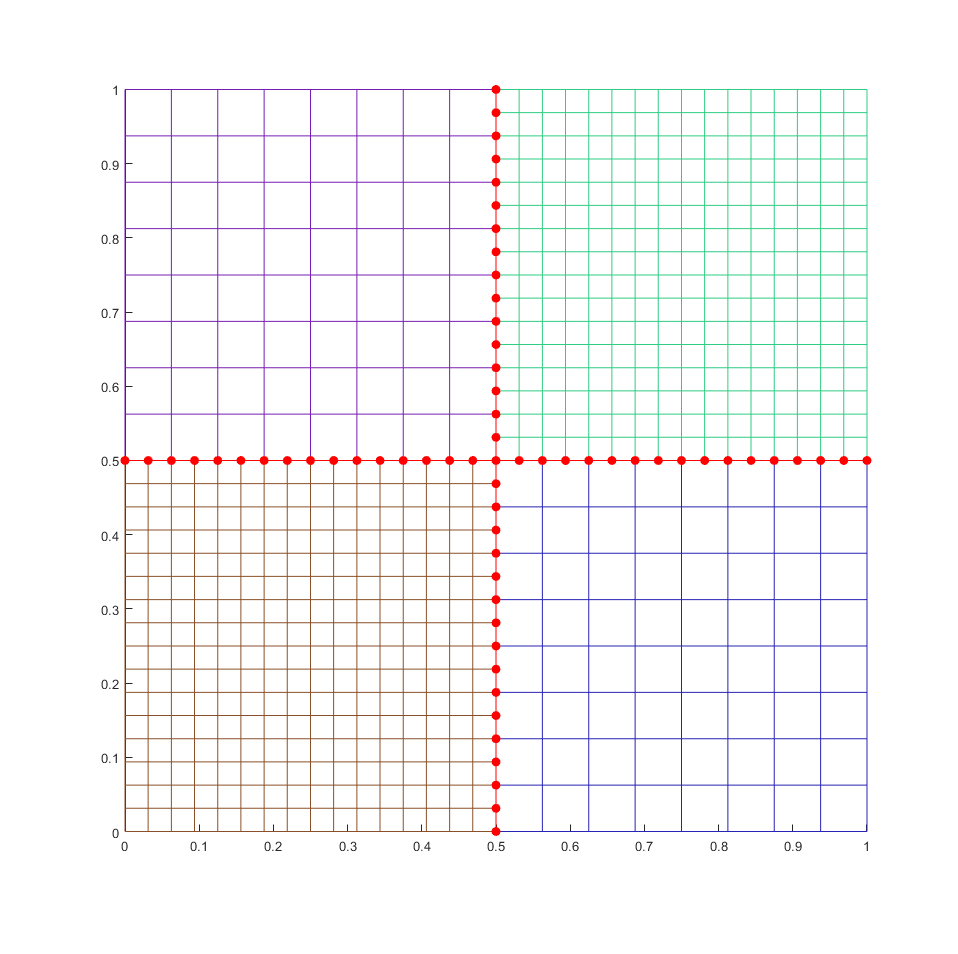}
	\includegraphics[width=0.3\linewidth]{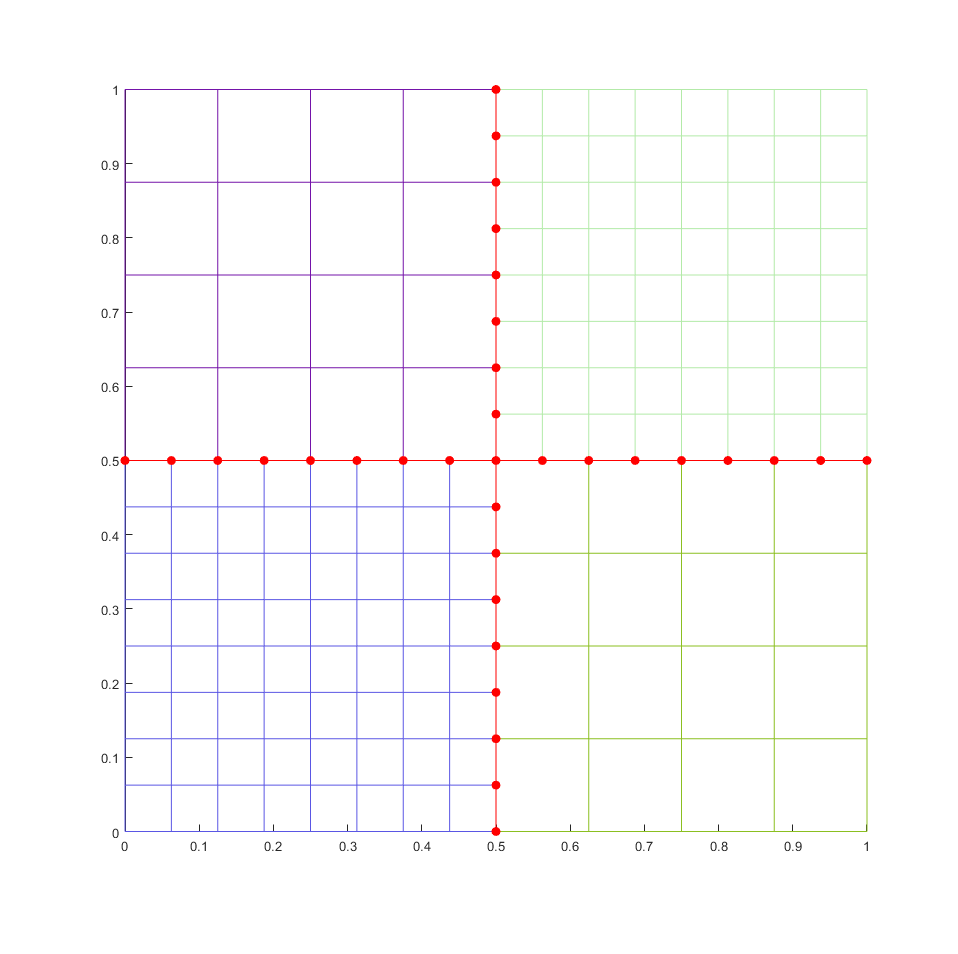}
	\caption{Example of non-matching grids for subdomains.}
	\label{fig:mesh_discretization}
\end{figure}

\begin{table}[H]
	\center
	\begin{tabular}{l|lll|r|r}
		\hline
		Level & $h$ & $H$ & $\Delta t$ & \multicolumn{1}{ c| }{$error_p$}  & $error_{\mathbf{u}}$ \\
		\hline
		1 &1/52 & 1/26 & 1/50 & 6.33e-04  & 1.51e-01  \\
		\hline
		2 & 1/100 & 1/50 & 1/100 & 3.32e-04     & 1.02e-01   \\
		\hline
		3 & 1/120 & 1/60 & 1/120 & 2.79e-04  & 9.15e-02  \\
		\hline
		4 & 1/152 & 1/76 & 1/150 & 2.26e-04   &  7.97e-02  \\
		\hline
	\end{tabular}
	\caption {Accuracy results of pressure and velocity for various levels.} \label{tab:errortimeN1} 
\end{table}

\begin{figure}[H]
	\centering
	\includegraphics[width=0.45\linewidth]{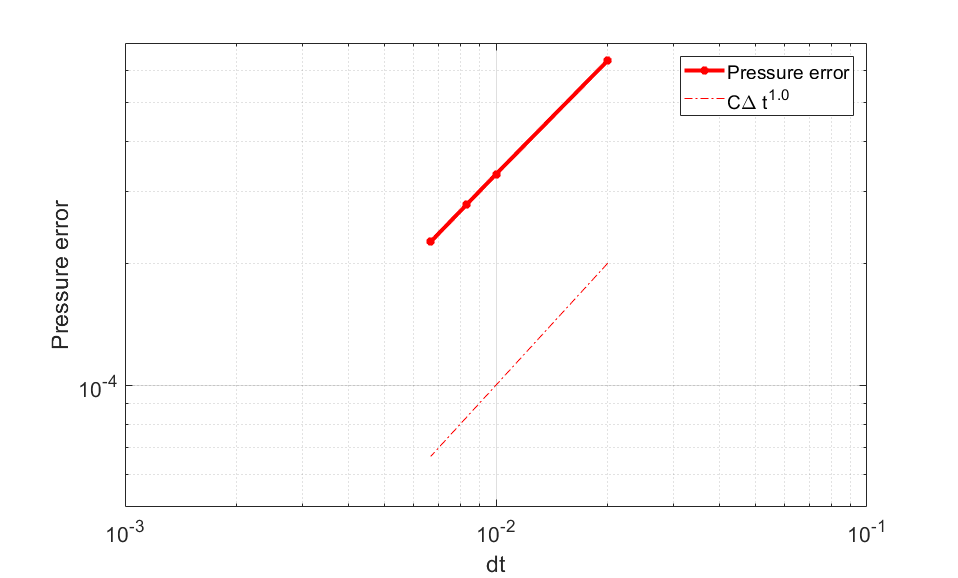}
	\includegraphics[width=0.45\linewidth]{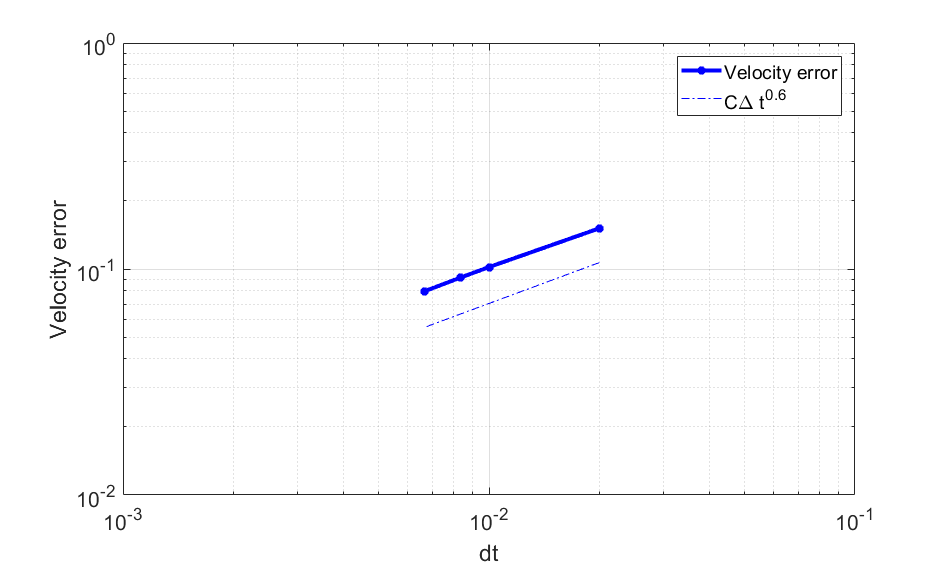}
	\caption{Convergence of the pressure and velocity error.}
	\label{fig:testN4_error_pres}
\end{figure}

\subsection{Numerical example 2}
We use manufactured the known solution
\begin{align*}
p(x,y,t)= e^{t}\sin(2\pi x)\sin(2\pi y)
\end{align*}
and use it to compute the forcing $f$, the Dirichlet boundary data $g$, and the initial data $p_0$. The time step is equal to the root of the coarse mesh size. Thus first order convergence is expected from theoretical result. 
Such mesh discretization are depicted in the Figure \ref{fig:mesh_discretization}. The simulation time interval is $(0, 2)$, i.e. $T=2$, and we use the Backward Euler method to integrate with regard to time with uniform time
step, see Table \ref{tab:errortimeN3}. The convergence rate is illustrated in Figure \ref{fig:testN3_error_vel}.

\begin{table}[H]
	\center
	\begin{tabular}{l|lll|r|r}
		\hline
		Level & $h$ & $H$ & $\Delta t$ & \multicolumn{1}{ c| }{$error_p$}  & $error_{\mathbf{u}}$ \\
		\hline
		1 &1/100 & 1/50 & 1/7 & 7.28e-01  & 8.64e-01 \\
		\hline
		2 & 1/128 & 1/64 & 1/8 &  6.38e-01   & 7.71e-01   \\
		\hline
		3 & 1/164 & 1/82 & 1/9 & 5.69e-01 & 6.87e-01   \\
		\hline
		4 & 1/200 & 1/100 & 1/10 & 5.13e-01 & 6.24e-01 \\
		\hline
	\end{tabular}
	\caption {Accuracy results of pressure and velocity for various levels.} \label{tab:errortimeN3} 
\end{table}

\begin{figure}[H]
	\centering
	\includegraphics[width=0.45\linewidth]{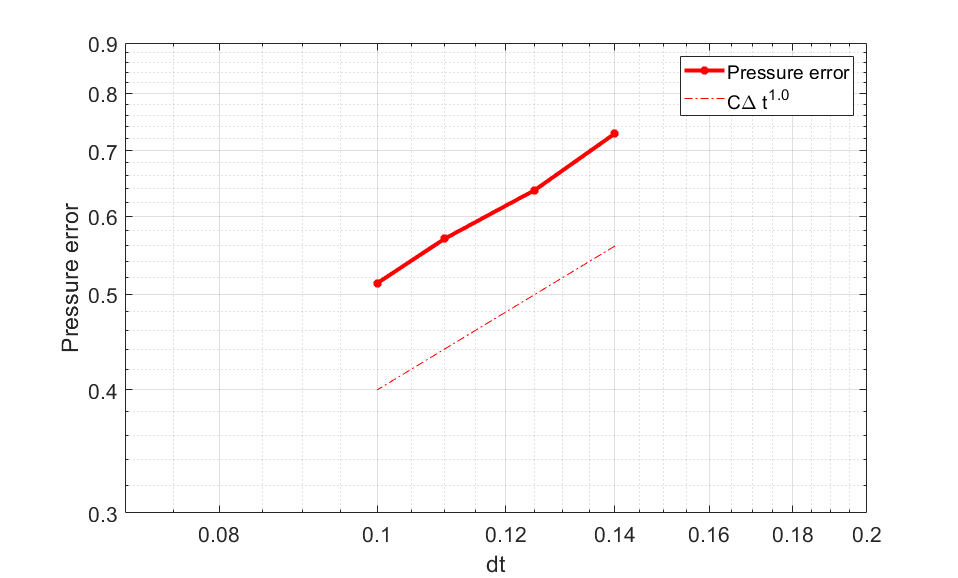}
	\includegraphics[width=0.45\linewidth]{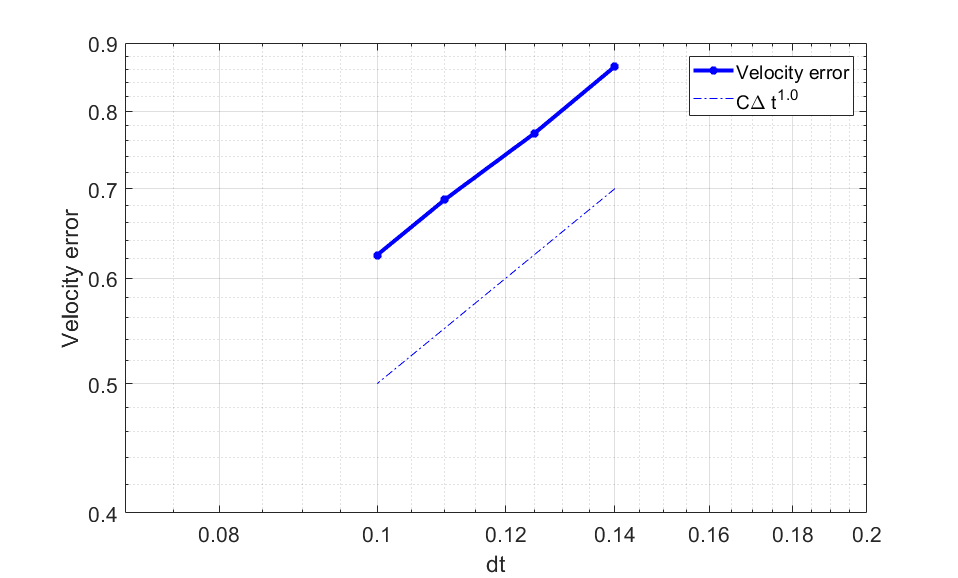}
	\caption{Convergence of pressure and velocity error.}
	\label{fig:testN3_error_vel}
\end{figure}


\section{Conclusion}
This research has provided a \textit{priori} error analysis for transient problems or slightly compressible flow problems through the heterogeneous porous media using Enhanced Velocity scheme as the domain decomposition method in space that coupled with backward Euler or Crank-Nicolson method in the time setting. In these discretization settings, we obtained the first order convergence rate for the backward Euler method and the second order convergence rate for the Crank-Nicolson method. Numerical experiments are provided.   The results suggest that this approaches could also be useful for the engineering subsurface applications including CO$_2$ sequestration, etc. In our future research, we plan to concentrate on parareal algorithms to achieve efficiency in time discretization that allow to run simulation efficiently for the long time range.
 
\section*{Acknowledgements} 
First author thanks Drs. T. Arbogast and I. Yotov for some helpful discussions during analysis of method.

\bibliographystyle{unsrt}  
\bibliography{priori_error_manuscript}

\begin{thebibliography}{10}

\bibitem{amanbek2017adaptive}
Yerlan Amanbek, Gurpreet Singh, Mary~F Wheeler, and Hans van Duijn.
\newblock Adaptive numerical homogenization for upscaling single phase flow and
  transport.
\newblock {\em ICES Report}, 12:17, 2017.

\bibitem{singh2017adaptive}
Gurpreet Singh, Yerlan Amanbek, and Mary~F Wheeler.
\newblock Adaptive homogenization for upscaling heterogeneous porous medium.
\newblock In {\em SPE Annual Technical Conference and Exhibition}. Society of
  Petroleum Engineers, 2017.

\bibitem{wheeler2002enhanced}
John~A Wheeler, Mary~F Wheeler, and Ivan Yotov.
\newblock Enhanced velocity mixed finite element methods for flow in multiblock
  domains.
\newblock {\em Computational Geosciences}, 6(3-4).

\bibitem{thomas2011enhanced}
Sunil~G Thomas and Mary~F Wheeler.
\newblock Enhanced velocity mixed finite element methods for modeling coupled
  flow and transport on non-matching multiblock grids.
\newblock {\em Computational Geosciences}, 15(4):605--625, 2011.

\bibitem{ganis2018adaptive}
Benjamin Ganis, Gergina Pencheva, and Mary~F Wheeler.
\newblock Adaptive mesh refinement with an enhanced velocity mixed finite
  element method on semi-structured grids using a fully coupled solver.
\newblock {\em Computational Geosciences}, pages 1--20, 2018.

\bibitem{amanbek2018new}
Yerlan Amanbek.
\newblock {\em A new adaptive modeling of flow and transport in porous media
  using an Enhanced Velocity scheme}.
\newblock PhD thesis, 2018.

\bibitem{singh2018space}
Gurpreet Singh and Mary~F Wheeler.
\newblock A space time domain decomposition approach using enhanced velocity
  mixed finite element method.
\newblock {\em arXiv preprint arXiv:1802.05137}, 2018.

\bibitem{singh2018domain}
Gurpreet Singh and Mary~F Wheeler.
\newblock A domain decomposition approach for local mesh refinement in space
  and time.
\newblock {\em arXiv preprint arXiv:1806.10187}, 2018.

\bibitem{Amanbek2018}
Yerlan Amanbek, Gurpreet Singh, and Mary~F. Wheeler.
\newblock Selective time-stepping adaptivity for non-linear reactive transport
  problems.
\newblock SIAM CSE 17. \url{https://doi.org/10.6084/m9.figshare.4702549.v4},
  2017.

\bibitem{thomee1984galerkin}
Vidar Thom{\'e}e.
\newblock {\em Galerkin finite element methods for parabolic problems}, volume
  1054.
\newblock Springer, 1984.

\bibitem{peaceman1983interpretation}
Donald~W Peaceman.
\newblock Interpretation of well-block pressures in numerical reservoir
  simulation with nonsquare grid blocks and anisotropic permeability.
\newblock {\em Society of Petroleum Engineers Journal}, 23(03):531--543, 1983.

\bibitem{wheeler1973priori}
Mary~F Wheeler.
\newblock A priori l\_2 error estimates for galerkin approximations to
  parabolic partial differential equations.
\newblock {\em SIAM Journal on Numerical Analysis}, 10(4):723--759, 1973.

\bibitem{riviere2000discontinuous}
B{\'e}atrice Rivi{\`e}re and Mary~F Wheeler.
\newblock A discontinuous galerkin method applied to nonlinear parabolic
  equations.
\newblock In {\em Discontinuous Galerkin methods}, pages 231--244. Springer,
  2000.

\bibitem{russell1983finite}
Thomas~F Russell and Mary~Fanett Wheeler.
\newblock {\em Finite element and finite difference methods for continuous
  flows in porous media}, pages 35--106.
\newblock SIAM, 1983.

\end{thebibliography}

\end{document}